\documentclass{article} 
\usepackage{amsmath,amsthm}     
\usepackage{graphicx}     
\usepackage{hyperref} 
\usepackage{url}
\usepackage{amsfonts} 
\usepackage{multicol}
\usepackage[utf8]{inputenc}

\newtheorem{theorem}{Theorem}
\newtheorem{lemma}{Lemma}

\theoremstyle{definition}
\newtheorem*{definition}{Definition}

\begin{document}

\title{An asymptotic Formula for the iterated exponential Bell Numbers}

\author{Ivar Henning Skau\\               
University of South-Eastern Norway\\    
3800 B{\o}, Telemark\\              
ivar.skau@usn.no
\and
Kai Forsberg Kristensen\\            
University of South-Eastern Norway\\    
3918 Porsgrunn, Telemark\\          
kai.f.kristensen@usn.no}                   

\maketitle

\begin{abstract}
In 1938 E. T. Bell introduced "The Iterated Exponential Integers". He proved that these numbers may be expressed by polynomials with rational coefficients. However, Bell gave no formulas for any of the coefficients except the trivial one, which is always 1. Our task has been to find the coefficient of the leading term, giving asymptotic information about these numbers.   
\end{abstract}

\section{Higher order Bell numbers}
The iterated exponential numbers, also called higher order Bell numbers, were introduced by E. T. Bell in \cite{B38}: 

\begin{definition}
The $m$-th order Bell numbers, $B_n^{(m)}\quad (m,n=0,1,\ldots)$, are given by the exponential generating functions 
$$E_m(x)=\sum\limits_{n=0}^\infty B_n^{(m)}\frac{x^n}{n!}\quad (m\geq 0),$$
where $E_0(x)=\exp(x)$ and $E_{m+1}(x)=\exp(E_m(x)-1)\quad (m\geq 0)$.
\end{definition}

\noindent Obviously we have $B_n^{(0)}=1$ and $B_0^{(m)}=E_m(0)=1$. In Table \ref{tab:bell} $B_n^{(m)}$ is computed for a few values of $m$ and $n$.\\

\begin{table}[htbp]
\begin{center}
\begin{tabular}{c||c|c|c|c|c|c|c|c}
$m\backslash n$& 1&2 &3&4&5&6&7&8\\
\hline
\hline
1&1& 2& 5& 15& 52& 203& 877& 4140\\
2&1& 3& 12& 60& 358& 2471& 19302& 167894\\
3&1& 4& 22& 154& 1304& 12915& 146115& 1855570\\
4&1& 5& 35& 315& 3455& 44590& 660665& 11035095\\
5&1& 6& 51& 561& 7556& 120196& 2201856& 45592666
\end{tabular}
\end{center}
\caption{Higher order Bell numbers $B_n^{(m)}$ when $1\leq n\leq 8$ and $1\leq m\leq 5$}
\label{tab:bell}
\end{table}

\noindent We note that $E_1(x)=\exp(\exp(x)-1)$ is the exponential generating function of the first order Bell numbers $B_n^{(1)}$, representing the total number $B_n$ of partitions of an $n$-set (see for example \cite[p. 24]{W06}). We also point out that the Stirling numbers of the second kind, $S(n,k)$, represent the number of $k$-partitions of an $n$-set, so that

\begin{equation}
	B_n=\sum_{k=1}^n S(n,k).
		\label{eq:bellrelation1}
\end{equation}

\noindent In \cite[p. 544]{B38} E. T. Bell proved a generalization of \eqref{eq:bellrelation1} that connects higher order Bell numbers to Stirling numbers of the second kind by the following recursion relation: 

\begin{equation}
	B_n^{(m)}=\sum_{k=1}^nB_k^{(m-1)}S(n,k). 
	\label{eq:bellrelation2}
\end{equation}
 
\noindent We will find \eqref{eq:bellrelation2} useful in the following section. 

\section{Polynomial expansions}

The introduction of higher order Bell numbers does not immediately suggest the existence of a polynomial representation. Still, that was exactly what Bell was able to prove in \cite[p.545]{B38}. Based on \eqref{eq:bellrelation2}, we will carry out a somewhat simpler proof, containing a few useful details: 

\begin{lemma}
$B_n^{(m)}$ may be expressed in the form 
\begin{equation}
B_n^{(m)}=c_{n-1}m^{n-1}+c_{n-2}m^{n-2}+\cdots+c_0,
\label{eq:bell_lin_comb}	
\end{equation}
where $c_{n-1},\ldots,c_0$ are rational numbers, independent of $m$. 
\end{lemma}

\begin{proof}
Our proof is by induction. First we observe that $B_1^{(m)}=E_m'(0)=\prod_{k=0}^mE_k(0)=1$. Together with $B_n^{(0)}=1$, this means that we must have $c_0=1$ if \eqref{eq:bell_lin_comb} is to be correct. Now, suppose that \eqref{eq:bell_lin_comb} is true for $k<n$, so that $B_k^{(m)}$ is a polynomial of degree $k-1$ with rational coefficients and constant term equal to 1. By rewriting \eqref{eq:bellrelation2}, we have 

\begin{equation}
	\begin{array}{lll}B_n^{(m)}-B_n^{(m-1)}&=&\displaystyle\sum\limits_{k=1}^{n-1}B_k^{(m-1)}S(n,k)\\[0.4cm]
	
	{}&=&d_{n-2}m^{n-2}+d_{n-3}m^{n-3}+\cdots+d_1m+d_0,\end{array}
	\label{eq:bell_induction}
\end{equation}

\noindent where the $d_i$'s are rational coefficients. Summing over $m$, the telescoping property gives

\begin{equation}
\sum\limits_{k=1}^m(B_n^{(k)}-B_n^{(k-1)})=B_n^{(m)}-1=\sum\limits_{r=0}^{n-2}d_r\left(\sum\limits_{k=1}^m k^r\right).	
\label{eq:telescoping}
\end{equation}

\noindent It is well known (see \cite[p. 525]{K90}) that  

\begin{equation}
	\sum\limits_{k=1}^m k^r=\frac{1}{r+1}m^{r+1}+m\sum\limits_{k=2}^{2\lfloor r/2\rfloor}\frac{b_k}{k}\binom{r}{k-1}m^{r-k}, 
	\label{eq:powers_of_integers}
\end{equation}

\noindent where the $b_k$'s are the (rational) Bernoulli numbers. This completes the induction step and proves the lemma.
\end{proof}

\section{An asymptotic formula of the higher order Bell numbers}

When $n$ is fixed, it is straightforward to see that $\lim_{m\rightarrow\infty}B_n^{(m)}/B_n^{(m-1)}=1$. However, looking more carefully into the details of the proof of the lemma, we can sharpen the asymptotics:

\begin{theorem}
Assume $n$ fixed. Then we have 
\begin{equation}
	B_n^{(m)}\sim \frac{n!}{2^{n-1}}m^{n-1}\quad (m\rightarrow\infty).
	\label{eq:bell_m_asymptotic}
\end{equation}
\label{theorem:asympt_m}
\end{theorem}

\begin{proof}
Since we already have the polynomial expansion from the lemma, it remains to prove that $c_{n-1}=n!/2^{n-1}$. To avoid ambiguity we will use the notation $c_{n-1}=c_{n-1}^{(n)}$ for the leading coefficient of $B_n^{(m)}$. By \eqref{eq:telescoping} and \eqref{eq:powers_of_integers} $c_{n-1}^{(n)}$ can be expressed as $d_{n-2}/(n-1)$. At the same time, because of \eqref{eq:bell_induction}, $d_{n-2}$ must be the $m^{n-2}$-coefficient of $B_{n-1}^{(m-1)}S(n,n-1)$, namely $c_{n-2}^{(n-1)}S(n,n-1)$. Since $S(n,n-1)=\binom{n}{2}=n/2$, this means that 
$$c_{n-1}^{(n)}=\frac{n}{2}c_{n-2}^{(n-1)}.$$
The initial value of this recurrence relation is $c_0^{(1)}=B_1^{(m)}=1$, which enables us to conclude that $c_{n-1}=c_{n-1}^{(n)}=n!/2^{n-1}$, proving \eqref{eq:bell_m_asymptotic}.
\end{proof}

\noindent In Table \ref{tab:bell_asymptotic_m} we see how the $m$-growth of $B_n^{(m)}$ is "explained" by the leading term $(n!/2^{n-1})m^{n-1}$ when $n=3$.

\begin{table}[htbp]
\begin{center}
\begin{tabular}{c||c|c|c}
$m$&$10^2$&$10^5$&$10^8$\\
\hline
&&&\\[-0.3cm]
$B_3^{(m)}$&15251&15000250001&15000000250000001\\
\hline
&&&\\[-0.2cm]
$(3!/2^{3-1})m^{3-1}$&15000&15000000000&15000000000000000

\end{tabular}
\end{center}
\caption{$B_3^{(m)}$ and $(3!/2^{3-1})m^{3-1}$ are compared. The computations are done with the aid of Maple 2016.}
\label{tab:bell_asymptotic_m}
\end{table}

\end{document}